  \documentclass[12pt]{amsart}
\address{Humboldt Universit\"at zu Berlin, Institut f\"ur Mathematik. Rudower Chaussee 25, 12489 Berlin, Germany.}

  \email{vasirog[at]gmail.com}

\usepackage[utf8]{inputenc}
\usepackage[T2A]{fontenc}
\usepackage{enumerate, amssymb, amsmath, amsthm}
\usepackage{amscd}
\usepackage{graphicx}
\usepackage{hyperref}
\hypersetup{
    colorlinks=true,
    linkcolor=blue,
    filecolor=blue,      
    urlcolor=blue,
    citecolor=magenta
}
 
\usepackage[a4paper,hmargin=2.5cm,vmargin=2.5cm]{geometry}
\usepackage{marginnote}
\usepackage[english]{babel}
\graphicspath{ {images/} }
\usepackage{wrapfig}
\usepackage[matrix,arrow,curve]{xy}\usepackage{mathabx}
\usepackage[OT2, T1]{fontenc}

\DeclareSymbolFont{cyrletters}{OT2}{wncyr}{m}{n}
\DeclareMathSymbol{\Sha}{\mathalpha}{cyrletters}{"58}

\usepackage[a4paper,hmargin=2.5cm,vmargin=2.5cm]{geometry}
\usepackage[english]{babel}
\graphicspath{ {images/} }
\usepackage{wrapfig}
\usepackage[matrix,arrow,curve]{xy}\usepackage{mathabx}

\begin{document}

\newtheorem{prop}{Proposition}[section]
\newtheorem{thrm}[prop]{Theorem}
\newtheorem{lemma}[prop]{Lemma}
\newtheorem{cor}[prop]{Corollary}
\newtheorem{mainthm}{Theorem}
\newtheorem{maincor}[mainthm]{Corollary}
\theoremstyle{definition}
\newtheorem{df}{Definition}
\newtheorem{ex}{Example}
\newtheorem{rmk}{Remark}
\newtheorem{conj}{Conjecture}
\newtheorem{cl}{Claim}
\newtheorem{q}{Question}
\renewcommand{\proofname}{\textnormal{\textbf{Proof:  }}}
\renewcommand{\refname}{Bibliography}
\renewcommand{\themainthm}{\Alph{mainthm}}
\renewcommand{\themaincor}{\Alph{maincor}}
\renewcommand{\C}{\mathbb C}
\newcommand{\Z}{\mathbb Z}
\newcommand{\Q}{\mathbb Q}
\newcommand{\R}{\mathbb R}
\newcommand{\N}{\mathbb N}

\renewcommand{\O}{\mathcal O}
\newcommand{\g}{\mathfrak g}
\newcommand{\D}{\mathcal D}

\renewcommand{\i}{\sqrt{-1}}
\renewcommand{\o}{\otimes}
\newcommand{\di}{\partial}
\newcommand{\acts}{\lefttorightarrow}
\newcommand{\dibar}{\overline{\partial}}
\newcommand{\im}{\operatorname{im}}
\renewcommand{\ker}{\operatorname{ker}}
\newcommand{\coker}{\operatorname{coker}}
\newcommand{\Hom}{\operatorname{Hom}}
\newcommand{\tr}{\operatorname{tr}}
\newcommand{\Alb}{\operatorname{Alb}}
\newcommand{\Hilb}{\operatorname{Hilb}}
\newcommand{\CP}{\mathbb{C}\mathbf{P}}
\newcommand{\Isom}{\operatorname{Isom}}
\newcommand{\Sym}{\operatorname{Sym}}

\newcommand{\GL}{\operatorname{GL}}
\newcommand{\SL}{\operatorname{SL}}
\newcommand{\SU}{\operatorname{SU}}
\renewcommand{\U}{\operatorname{U}}
\newcommand{\SO}{\operatorname{SO}}
\newcommand{\Ogr}{\operatorname{O}}
\newcommand{\Sp}{\operatorname{Sp}}

\newcommand{\codim}{\operatorname{codim}}
\newcommand{\rk}{\operatorname{rk}}
\newcommand{\hdot}{{\:\raisebox{3pt}{\text{\circle*{1.5}}}}}
\newcommand{\alg}{\mathrm{alg}}
\newcommand{\an}{\mathrm{an}}
\newcommand{\tM}{\widetilde{M}}
\newcommand{\tg}{\widetilde{g}}
\renewcommand{\S}{\mathcal S}
\binoppenalty = 10000
\relpenalty = 10000

\title{On the metric Koll\'ar-Pardon problem}

\author{Vasily Rogov}

\begin{abstract}
Let $(M, g)$ be a compact real analytic Riemannian manifold and $\pi \colon \tM \to M$ its universal cover. Assume that $\tM$ can be realised as a manifold definable in an o-minimal structure $\Sigma$ expanding $\R_{\an}$ in such a way that the pullback metric $\widetilde{g}:=\pi^*g$ is $\Sigma$-definable. For instance, this is the case when $\tM$ can be realised as a  semi-algebraic submanifold in $\R^n$ in such a way that the coefficients of the metric $\widetilde{g}$ are semi-algebraic.  We show that there exists a definable smooth map $\tM \to \widetilde{K}$ to a compact simply connected $\Sigma$-definable space $\widetilde{K}$ such that its regular fibres are Riemann locally homogeneous with respect to the metric $\widetilde{g}$. We deduce that under these assumptions $\pi_1(M)$ is quasi-isometric to a locally homogeneous space. In the case when $M$ is aspherical we show that $(\tM, \widetilde{g})$ is a homogeneous  Riemannian manifold. A similar result in the setting of complex algebraic geometry was earlier conjectured by Koll\'ar and Pardon (\cite{KP}). Using our results, we prove the conjecture of Koll\'ar-Pardon in the special case of smooth aspherical varieties admitting a bi-definable K\"ahler metric and discuss the analogues of this conjecture in other branches of geometry.
\end{abstract}

\maketitle

\section{Introduction}\label{intro}

Let $\Gamma$ be a finitely generated group that acts by homeomorphisms properly discontinuously and cocompactly on an $n$-dimensional ball $B$. Choose any homeomorphism $\phi \colon B \xrightarrow{\sim} \R^n$. Let  $f \colon \R^n \to \R$ be a polynomial map. Assume that $f$ is $\Gamma$-invariant, that is $f(\phi(\gamma \cdot x))=f(\phi(x))$ for any $x \in B, \ \gamma \in \Gamma$. Then $f$ is constant. 

The aim of this note is to take a step towards a wide geometrical generalisation of this elementary fact. The geometric meta-principle we are aiming for can be formulated as the following:
\\

<<If a geometric structure demonstrates \emph{tame behaviour at infinity} and is invariant under a cocompact discrete group action, it is \emph{close to being homogenous}>>.
\\ 

 The notion of tameness can be made precise using the language of o-minimal geometry, see Section \ref{prelim}. The words <<\emph{close to being homogenous}>> should most probably mean <<homogenous up to a compact factor>>, see Theorem \ref{mainA} and the discussion there.

The first result in this direction was proposed by Koll\'ar and Pardon in their paper \cite{KP} and was motivated by the study of universal covers of complex algebraic varieties. Let $X$ be a smooth complex projective variety and $\pi \colon \widetilde{X} \to X$ its universal cover\footnote{By this we mean the universal cover of the underlying complex analytic space $X^{\an}$.}. The Shafarevich conjecture (\cite{Shaf}, p.407) predicts that $\widetilde{X}$ is holomorphically convex, which is a deep complex-analytic property. It is now known in many cases (see e.g. \cite{EKPR}), while the general case is still widely open. One of the difficulties of this conjecture is that the operation of taking the universal cover is purely transcendental: the complex manifold $\widetilde{X}$ usually does not possess an algebraic structure and lacks any reasonable description in terms of polynomial equations, or even real polynomial inequalities.  The conjecture of Koll\'ar and Pardon says that $\widetilde{X}$ is indeed never semi-algebraic, except for several classical cases.

\begin{conj}[Koll\'ar-Pardon, \cite{KP}]\label{KP conj}
Let $X$ be a smooth complex projective variety and $\pi \colon \widetilde{X} \to X$  its universal cover. Assume that $\widetilde{X}$ is biholomorphic to an open real semi-algebraic subset $\Omega$ inside a complex algebraic variety $Z$. Then $\widetilde{X} \simeq F \times \C^n \times D$, where $D$ is a symmetric bounded domain and $F$ is a  simply connected projective variety.
\end{conj}

Recall that a subset $\Omega$ of a complex algebraic variety $Z$ is real semi-algebraic if $Z$ admits a finite open cover by affine sets $U_i$ such that $\Omega \cap U_i$ is given by a finite collection of real polynomial equalities and inequalities for every $i$.

In the same paper, the authors proved their conjecture in some special case (Theorem 3, \cite{KP}). In the earlier work \cite{CHK} some results were obtained when $\Omega=Z$ is quasi-projective. See also \cite{CH} for some results when $X$ is compact K\"ahler and $\Omega$ is quasi-K\"ahler.

Let us clarify the relation between Conjecture \ref{KP conj} and the meta-principle that we mentioned above. The complex structure on $\Omega$ is \emph{semi-algebraic}. This means that there exists a finite collection of real semi-algebraic charts $\Omega = U_1 \cup \ldots \cup U_k$ and real algebraic trivialisations of the tangent bundle $T\Omega|_{U_i}$, in which the complex structure operator $J_{\Omega}$ is given by semi-algebraic maps $\phi_i \colon U_i \to \operatorname{Mat}_{n \times n}(\R)$, where $n=\dim X$. On the other hand, the holomorphic action of $\pi_1(X)$ on $\widetilde{X}$ translates naturally to $\Omega$ and $J_{\Omega}$ is $\pi_1(X)$-invariant. Conjecture \ref{KP conj} therefore predicts that up to a compact factor $F$ (which is, in fact, irrelevant to the $\pi_1(X)$-structure on $\widetilde{X}$), the complex manifold $\widetilde{X}$  is biholomorphic to the complex homogeneous manifold $\C^{n} \times D$. Moreover, $\widetilde{X}$ is K\"ahler since it is locally biholomorphic to the K\"ahler manifold $X$, and every complete K\"ahler homogenous simply connected manifold is of this form (see e.g. \cite{Ghys}, Th\'eor\`eme 2.5).

A set-up for Koll\'ar-Pardon conjecture which is more natural and general than semi-algebraic geometry is \emph{o-minimal geometry} (see \cite{KP}, Remark 18). We briefly recall the basics of o-minimal geometry in Section \ref{prelim} and refer the reader to \cite{vdD} for the comprehensive introduction. The main result of this paper is the following Riemannian version of Conjecture \ref{KP conj}:

\begin{mainthm}\label{mainA}[Theorem \ref{thmA}]
Let $M$ be a compact real analytic manifold and $g$ a real-analytic Riemannian metric on $M$. Let $\pi \colon \tM \to M$ be the universal cover and $\widetilde{g}:=\pi^*g$  the induced metric on it. Assume that $\tM$ admits a structure of a \emph{manifold definable in an o-minimal structure} $\Sigma$ in such a way that $\widetilde{g}$ is $\Sigma$-definable. Then there exists a $\Sigma$-definable simply connected compact topological space $\widetilde{K}$ and a $\Sigma$-definable real analytic map $\Phi \colon  \tM \to \widetilde{K}$, such that the regular fibres of $\Phi$ are Riemannian locally homogeneous manifolds with respect to the metric $\widetilde{g}$.
\end{mainthm}

In fact, we prove a stronger result: for every point $x \in \widetilde{M}$ which is not a critical point of $\Phi$, there exists a connected Lie group $G$ and an open neighbourhood $U$ of $x$, such that $G$ locally acts on $(U, \widetilde{g})$ by isometries and the fibres of $\Phi$ are given by the orbits of this action.

 If the map $\Phi$ has no critical points, Theorem \ref{mainA} implies that $(\widetilde{M}, \widetilde{g})$ is fibred in locally homogeneous submanifolds over a compact topological space. In this case, we get a Riemannian counterpart to the Conjecture \ref{KP conj}. We provide some examples where the map $\Phi$ has critical points, see Example \ref{revolution}. In this case, the best one can hope for is to find a stratification of $\widetilde{M}$ by $\Sigma$-definable locally closed subsets, each of which is fibred with locally homogeneous fibres. We show that this is indeed the case (Proposition \ref{strata}) and that $M$ can be written as a union of closed locally homogeneous submanifolds which are \emph{strongly bi-definable} in the sense of Definition \ref{bi-def}, see Corollary \ref{union}.

 We also deduce that a compact manifold $M$ satisfies the assumptions of Theorem \ref{mainA}, then $\pi_1(M)$ is quasi-isometric to a locally homogeneous space, which seems to be a subtle geometric-group-theoretic property (Corollary \ref{quasi-isometric}).

If $M$ is aspherical, Theorem \ref{mainA} can be strengthened as the following.

\begin{mainthm}\label{mainB}[Theorem \ref{homogene}]
Assume $M, \ g$ and $\tM$ are as in Theorem \ref{mainA}. Assume moreover that $\tM$ is contractible. Then $\widetilde{K}$ is a point and $(\tM, \widetilde{g})$ is a Riemannian homogeneous manifold.
\end{mainthm}

The proof of Theorem \ref{mainB} is based on a topological lemma (Lemma \ref{podkorytov}) that gives control over submanifolds in aspherical manifolds with ''topologically tame'' preimage in the universal cover. One of the general consequences of this lemma in o-minimal geometry is that compact aspherical manifolds with definable universal covers do not admit strict submanifolds that are \emph{strongly be-definable} (see Corollary \ref{no bi-def}; Definition \ref{bi-def} for the notion of strongly be-definable submanifolds).

From this, we deduce an application to the original Koll\'ar-Pardon conjecture.

\begin{mainthm}[Corollary \ref{KP special}]\label{mainC}
Let $X$ be a smooth complex projective variety with universal cover $\pi \colon \widetilde{X} \to X$. Assume that $\widetilde{X}$ admits a structure of a definable complex manifold in some o-minimal structure. Assume moreover that  $X$ is aspherical and admits a \emph{bi-definable K\"ahler metric}, that is, a K\"ahler Hermitian metric $h=g + \sqrt{-1}\omega$ such that $\pi^*h$ is definable on $\widetilde{X}$. Then $(\widetilde{X}, \pi^*h)$ is a K\"ahler homogeneous manifold. In particular, the Koll\'ar-Pardon Conjecture holds for $X$.
\end{mainthm}

The paper is organised as follows. In Section \ref{prelim} we briefly recall the necessary preliminaries. We discuss o-minimal structures and o-minimal geometry in subsection \ref{o-min} and \ref{o-min top} and translate the classical differential-geometric notions to the context of o-minimal geometry in subsection \ref{o-min diff}. In subsection \ref{Singer part} we recall the relation between cohomogeneity and additive Weyl invariants after Console-Olmos (\cite{CO}), a blackbox from Riemannian geometry which plays a key role in the proof. In Section \ref{main part} we introduce the concept of (strong) bi-definability (subsection \ref{bi-def part}). We prove Theorem \ref{mainA} in subsection \ref{proof a part}, Lemma \ref{podkorytov} and Theorem \ref{mainB} in subsection \ref{proof b part}, and Theorem \ref{mainC} in subsection \ref{kp part}. We finish with a discussion on some open questions (subsection \ref{questions}).
\\

\textit{Acknowledgements.} I am thankful to Bruno Klingler for turning my attention to the Koll\'ar-Pardon Conjecture and several fruitful discussions. I am also thankful to Paul Brommer-Wierig for interesting discussions and to Semyon Podkorytov for sharing with me the main ideas behind  Lemma \ref{podkorytov}.

\section{Preliminaries}\label{prelim}

\subsection{o-minimal geometry}\label{o-min} The notion of an o-minimal structure originated in model theory and  found many applications  in other areas of mathematics during the last decade, specifically in complex algebraic geometry (\cite{BKT}, \cite{BBT}, etc.) In this section, we briefly recall the basics of o-minimal geometry and mention all the needed geometric properties of o-minimally definable sets. The primary reference for this part is \cite{vdD}.

\begin{df}
A \emph{structure} $\Sigma$ is a collection $(\Sigma_n)_{n \in \N}$, where each $\Sigma_n$ is a boolean algebra of subsets of $\R^n$ and the following holds:
\begin{itemize}
\item[(i)] if $A \in \Sigma_n$ and $B \in \Sigma_m$, then $A \times B \in \Sigma_{n+m}$;
\item[(ii)] if $A \in \Sigma_n$ and $p_i \colon \R^{n} \to \R^{n-1}$ is the standard projection $$(x_1, \ldots, x_n) \mapsto (x_1, \ldots x_{i-1},x_{i+1}  \ldots, x_{n-1}),$$ then $p_i(A) \in \Sigma_{n-1}$;
\item[(iii)] every semi-algebraic subset $A \subseteq \R^n$ is in $\Sigma_n$. Recall, that $A$ is semi-algebraic if it can be defined using a finite system of real polynomial inequalities.
\end{itemize}
The elements of $\Sigma_n$ are called ($\Sigma$-)\emph{definable sets}.
If $A \in \Sigma_n$ and $B \in \Sigma_m$ are two definable sets, then a map $f \colon A \to B$ is called \emph{definable} if its graph $\Gamma_f \subseteq A \times B$ is definable.
\end{df}

The most elementary example is given by the structure of semi-algebraic sets $\R_{\alg}$.  Notice that the property \textit{(ii)} for $\R_{\alg}$ is a consequence of the Tarski-Seidenberg Quantifier Elimination Theorem.

If $\Sigma$ is a structure, images and preimages of $\Sigma$-definable sets under $\Sigma$-definable maps are $\Sigma$-definable. Moreover, if $A \subseteq \R^n$ is $\Sigma$-definable then its topological closure $\overline{A}$, interior $A^{\circ}$, boundary $\di A = \overline{A} \setminus A$ and frontier $\overline{A} \setminus A^{\circ}$ are $\Sigma$-definable (see e.g. \cite{vdD}, Chapters 1 and 2).

\begin{df}
A structure $\Sigma$ is called \emph{o-minimal}, if $\Sigma_1=(\R_{\alg})_1$. Equivalently, the only one-dimensional $\Sigma$-definable sets are finite collections of points, intervals and rays.
\end{df}

The sets definable in $\R_{alg}$ share many nice topological properties, and restricting oneself to working with semi-algebraic sets is a safe way to escape pathological counterexamples from point-set topology. What is spectacular, is that the same feature remains for any o-minimal structure. Below we list the main geometric properties of sets definable in o-minimal structures.

\begin{thrm}\label{o-min nice}
Let $\Sigma$ be an o-minimal structure.
\begin{itemize}
\item[(i)] A $\Sigma$-definable set is connected if and only if it is path-connected;
\item[(ii)] definable sets admit \emph{finite cell decomposition}: every definable $A \subset \R^n$ can be written as $A = A_1 \sqcup \ldots \sqcup A_k$, where each $A_j \subseteq A$ is a locally closed definable subset definably homeomorphic to $\R^{n_j}$. In particular, every definable set has finitely many connected components and is homotopically equivalent to a finite CW-complex;
\item[(iii)] definable functions are piecewise-continuous; moreover, if $A$ is a definable set and $f \colon A \to \R$ is a definable map, for every $k \ge 0$ there exists a finite decomposition $A = A_1 \sqcup \ldots \sqcup A_m$ into definable locally closed sets $A_j$ such that $f|_{A_j}$ is of class $\mathcal{C}^{k}$;
\item[(iv)] if $A \subseteq \R$ is a definable subset of $\R$ and $f \colon A \to \R$ is a definable function, then its derivative $f'$ is definable whenever it exists.
\end{itemize}
\end{thrm}

Item \textit{(i)} shows that working with sets definable in o-minimal structure indeed allows one to exclude pathological examples that appear for arbitrary subsets of $\R^n$. All the results of Theorem \ref{o-min nice} can be found in \cite{vdD}. The item  \textit{(ii)} is  \cite{vdD}, Chapter 3, Theorem 2.11. The items \textit{(iii)} and \textit{(iv)} are the content of Chapter 7 \textit{loc. cit.}.

\begin{rmk} \label{smooth cell}
Item \textit{(ii)} of Theorem \ref{o-min nice} can be strengthened as the following. Let $\Sigma$ be an o-minimal structure and $A \subseteq \R^n$ a definable subspace.Then for every $k\ge 0$ there exists a decomposition of $A$ into a disjoint union of locally closed subsets
\[
A= A_1 \sqcup \ldots \sqcup A_k
\]
and a collection of maps $\theta_j \colon A_j \to \R^{n_j}$, where each $\theta_j$ is a definable homeomorphism of class $\mathcal{C}^k$. In particular, every $A_j$ is a $\mathcal{C}^k$-submanifold of $\R^n$. For the proof see Theorem 6.6 in \cite{Cos}. In \cite{RSW} the authors showed that one can replace $k$ with $\infty$ if $\Sigma$ is a  so-called \emph{structure generated by a quasianalytic Denjoy-Carleman class} (in what follows we will call it \emph{of QA type}, following the abbreviation in \cite{RSW}). The structures $\R_{alg}, \R_{an}, \R_{exp}$ and $\R_{an,exp}$ are all of QA type.
\end{rmk}

An example of a set which is not definable in any o-minimal structure is the graph of the sine function $\{(x, \sin x)\} \subset \R^2$. Indeed, the sine function can not be definable since  $\sin^{-1}(0)=\{n\pi| n \in \Z\} \subset \R$ is not in $(\R_{alg})_1$. Another example which is relevant to our goals is the following. Let $M_1$ and $M_2$ be definable subsets of $\R^n$ and $f \colon M_1 \to M_2$ an infinite-folded cover. Then the map $f$ is not definable in any o-minimal structure since a fibre over a point is a discrete countable set and can not be definable by Theorem \ref{o-min nice}, \textit{(ii)}.

O-minimal structures are in general very hard to construct. A structure $\Sigma'$ is called an \emph{expansion} of a structure $\Sigma$ if every $\Sigma$-definable set is $\Sigma'$-definable. If $\Sigma$ is a structure and $\mathcal{T}$ is a collection of subsets of $\R^n$, by Zorn's lemma there exists a minimal structure $\Sigma\langle \mathcal{T} \rangle$ which is an expansion of $\Sigma$ and contains $\mathcal{T}$. Of course, even if we start with an o-minimal structure $\Sigma$, its expansion $\Sigma \langle \mathcal{T} \rangle$ does not need to be o-minimal. Here is the list of main known o-minimal structures:
\begin{itemize}
\item $\R_{\alg}$;
\item $\R_{\mathrm{an}}$. This structure is defined as $\R_{\alg}\langle \mathcal{T} \rangle$, where $\mathcal{T}$ is the set of graphs of functions $f \colon \R^n \to \R$, where $f \equiv 0$ outside a compact $K \subset \R^n$ and $f|_K$ is analytic. The o-minimality of this structure was proved by van den Dries based on the result of Gabrielov \cite{Gab};
\item $\R_{\exp}$. This is structure is defined as $\R_{\alg} \langle  \Gamma_{\operatorname{exp}} \rangle$, where $\Gamma_{\operatorname{exp}}=\{(t, e^t)\} \subset \R^2$ is the graph of real exponent. The o-minimality is due to Wilkie \cite{Wil}.
\item $\R_{\mathrm{an}, \exp}$. This is the minimal structure which is an expansion of both $\R_{\mathrm{an}}$ and $\R_{\exp}$. Its o-minimality is proved in \cite{vdDM}.
\item The \emph{Pfaffian closure} of an o-minimal structure is an o-minimal structure, see \cite{Speis}.
\end{itemize}

\subsection{o-minimal spaces}\label{o-min top}

One would like to consider definable geometrical objects without fixing their embedding into $\R^n$. This motivates the following definition:

\begin{df}\label{dif top space}
Let $\Sigma$ be an o-minimal structure and $X$ a topological space. A $\Sigma$-\emph{definable atlas} $\xi=(U_{\alpha}, \phi_{\alpha})_{\alpha \in A}$ is a finite collection of open charts $U_{\alpha}, \ X = \bigcup_{\alpha} U_{\alpha}$ and  homeomorphic embeddings $\phi_{\alpha} \colon U_{\alpha} \to \R^{n}$, such that 
\begin{itemize}
\item[(i)] $\phi_{\alpha}(U_{\alpha} \cap U_{\beta})$ are $\Sigma$-definable in $\R^n$ for every $\alpha$ and $\beta$;
\item[(ii)] the transition functions $\phi_\beta \circ \phi_{\alpha}^{-1} \colon \phi_{\alpha}(U_{\alpha} \cap U_{\beta}) \to \phi_{\beta}(U_{\alpha} \cap U_{\beta})$ are $\Sigma$-definable. 
\end{itemize} 
A $\Sigma$-definable topological space is a topological space $X$ with a fixed equivalence class of $\Sigma$-definable atlases.
\end{df}

The finiteness of the atlas $\xi$ is crucial in Definition \ref{dif top space} since it allows one to extend geometrical and topological finiteness properties of definable sets to a broader context of definable spaces.

In what follows we will always assume that $\Sigma$ is an o-minimal structure which is an expansion of $\R_{an}$ and we will often not specify $\Sigma$ when talking about definable spaces. We all also abusively refer to a pair $(X, \xi)$ as \emph{definable space}, although one topological space a priori can admit non-equivalent definable structures.

If $X$ is a definable space, we say that a subspace $Y \subseteq X$ is definable, if there exists a finite definable atlas $(U_{\alpha}, \phi_{\alpha})$, such that $\phi_{\alpha}(Y \cap U_{\alpha})$ is definable for every $\alpha$. A product of two definable spaces carries a natural definable space structure, which allows us to speak about definable maps between definable spaces. All the properties listed in Theorem \ref{o-min nice} transfer directly to the context of definable spaces.  The Cellular Decomposition Theorem (Theorem \ref{o-min nice}, \textit{(ii)}) implies that a definable topological space is homeomorphic to a finite CW-complex.

Notice that in Definition \ref{dif top space} we do not require the sets $\phi_{\alpha}(U_{\alpha})$ to be smooth, therefore a definable topological space is not necessarily a manifold. One says that $M$ is a $\Sigma$-definable smooth (respectively $\mathcal{C}^k$-) manifold if it is a definable topological space and the definable atlas defines a smooth (respectively $\mathcal{C}^k$-) manifold structure on $M$. Again, a definable manifold is not a property of a smooth manifold, but a structure on it. Since in this paper we are not considering different definable structures on the same manifold, we abbreviate \emph{''definable manifold''} for ''a manifold admitting a structure of a manifold definable in an o-minimal structure $\Sigma$''.

In spite of the fact that not every definable topological space is a definable manifold, it admits a stratification by definable $\mathcal{C}^k$- (or sometimes even $\mathcal{C}^{\infty}$-) manifolds, as we explain below.

By Theorem \ref{o-min nice}, \textit{(ii)} there exists a decomposition $X= X_1 \sqcup \ldots \sqcup X_k$, where each $X_i$ is a locally closed definable subset of $X$ definably homeomorphic to $\R^{n_i}$. The \emph{dimension} of $X$ is defined as $\dim X:= \max \{n_i\}$ (by definition, the dimension of an empty set is $-\infty$). This notion does not depend on the choice of a cellular decomposition and gives rise to a nice dimension theory for definable topological spaces, see \cite{vdD}, Chapter 4.  In particular, if $A \subseteq X$ is an open definable subset of a  definable topological space, $\dim \di A < \dim A$.

\begin{lemma}\label{smooth strata}
Let $\Sigma$ be an o-minimal structure and $X$ be $\Sigma$-definable space. Then for every $k\ge0$, there exists a stratification by definable locally closed subsets
\[
\varnothing=S_0 \subset \ldots \subset S_r=X,
\]
such that  $\dim S_{i-1}<\dim S_{i}$ and each $V_i:=S_i\setminus S_{i-1}$ is a $\mathcal{C}^k$-definable manifold (with respect to the induced definable structure). If, moreover, $\Sigma$ is of QA type (see Remark \ref{smooth cell}), one can also take  $k=\infty$.
\end{lemma}
\begin{proof}
We argue by induction on $\dim X$. If $\dim X = 0$ there is nothing to prove. In general, from Theorem \ref{o-min nice}, \textit{(ii)} and Remark \ref{smooth cell} it follows that we can find a $\mathcal{C}^k$- (and $\mathcal{C}^{\infty}$-, if $\Sigma$ is of QA type) cellular decomposition
\[
X=X_1 \sqcup \ldots \sqcup X_p.
\]
Let $I=\{i_1, \ldots, i_q\}$ be the set of indices of cells of non-maximal dimension, that is  $\dim X_{i_j}<\dim X$ if and only if $i_j \in I$. Put $S_{r-1}:= \bigcup_{i\in I} X_i$. Then $V_r:=X \setminus S_{r-1}=\bigsqcup_{j \notin I} X_j$ is a union of non-interesecting definable $\mathcal{C}^k$- (respectively $\mathcal{C}^{\infty}$-) manifolds. At the same time, $\dim S_{r-1}< \dim X$, so the induction assumption applies to it.
\end{proof}

In this text, we are concerned with universal covers of various manifolds and spaces. Usually, the universal cover of a definable space does not possess any natural definable structure. An important exception is the case of a finite fundamental group.

Let  $X$ be a topological space homeomorphic to a finite CW-complex and $p \colon \widetilde{X} \to X$ be its universal cover. Let $f \colon Y \to X$ be a continuous map from a simply connected topological space $Y$. Recall that $f$ uniquely lifts to a continuous map $\widetilde{f} \colon Y \to \widetilde{X}$ such that $f=p \circ \widetilde{f}$.

\begin{lemma}\label{o-min cover}
Let $X$ be a topological space and $\xi$ a $\Sigma$-definable atlas on $X$ for an o-minimal structure $\Sigma$. Assume that $\pi_1(X)$ is finite. Then the following holds:
\begin{itemize}
\item[(i)] The universal cover $\widetilde{X}$ admits a unique $\Sigma$-definable atlas $\widetilde{\xi}$ such that the covering map $p \colon \widetilde{X} \to X$ is definable;
\item[(ii)] if $Y$ is a simply connected $\Sigma$-definable space and $f \colon Y \to X$ is continuous and $\Sigma$-definable, the canonical lifting $\widetilde{f} \colon Y \to \widetilde{X}$ is $\Sigma$-definable.
 \end{itemize}
\end{lemma}
\begin{proof}
Let $U_{\alpha} \in \xi$. By Theorem \ref{o-min nice}, \textit{(ii)} there exists a finite covering of $U_{\alpha}$ by open $\Sigma$-definable simply connected $V^{\alpha}_{\beta}, \ U_{\alpha} = \bigcup_{\beta} V^{\alpha}_{\beta}$. The cover $p \colon \widetilde{X} \to X$  splits over $V^{\alpha}_{\beta}$, that is, $p^{-1}(V^{\alpha}_{\beta})\simeq V^{\alpha}_{\beta} \times \{1, \ldots, k\}$. Here we use that $p^{-1}(V^{\alpha}_{\beta})$ has finitely many connected components since $\pi_1(X)$ is finite and $p$ is a finite cover. Each of these connected components is homeomorphic to $V^{\alpha}_{\beta}$. The  connected components of preimages of $V^{\alpha}_{\beta}$ for all $\alpha$ and $\beta$ give rise to  a definable atlas $\widetilde{\xi}$ on $\widetilde{X}$ with local coordinates $\phi_{\alpha} \circ p$. The projection $p$ is definable since in a finite collection of definable charts $V^{\alpha}_{\beta}$ it coincides with the projection $V^{\alpha}_{\beta} \times \{1, \ldots, k\} \to V^{\alpha}_{\beta}$.

To prove \textit{(ii)} notice that the graph $\Gamma_{\widetilde{f}} \subset Y \times \widetilde{X}$ is nothing but the preimage of the graph $\Gamma_f \subset Y \times X$ under the (definable) map $id \times p \colon Y \times \widetilde{X} \to Y \times X$. Therefore, $\widetilde{f}$ is definable whenever $f$ is.
\end{proof}

\subsection{o-minimal differential geometry} \label{o-min diff} In this subsection we transfer the main notions of differential geometry into the context of definable manifolds. These results are well-known to the experts, although folklore and, up to our knowledge, are not written in full generality anywhere in the literature.  We always work in the smooth category, although all the results remain valid in $\mathcal{C}^k$-category for $k\ge 2$.

\begin{df}
Let $M$ be a definable smooth manifold and $E \to M$ a smooth rank $d$ vector bundle. \emph{A definable structure on} $E$ is a finite definable cover $M= \bigcup U_{\alpha}$  (i.e., each $U_{\alpha}$ is open definable) and a collection of trivialisations $\psi_{\alpha} \colon E|_{U_{\alpha}} \xrightarrow{\sim} U_{\alpha} \times \R^d$, such that the transition functions $\psi_{\beta} \circ \psi_{\alpha}^{-1}$ are definable. We say that a  section $s \in \mathcal{C}^{\infty}(M, E)$ is definable if it is definable in the local definable charts. A definable vector bundle over a definable manifold is a vector bundle with a chosen definable structure.
\end{df}

A definable structure on a vector bundle $E$ naturally gives rise to a structure of a definable manifold on its total space such that the projection $p \colon \operatorname{Tot}(E) \to M$ is definable, and vice versa.  A section of a definable vector bundle $E \to M$ is definable if and only if it is definable as a map from $M$ to $\operatorname{Tot}(E)$. A trivial vector bundle has a canonical definable structure that comes from the product definable structure on its total space $M \times \R^d$.

A morphism $f \colon E \to F$ of definable vector bundles is definable if it sends definable sections of $E$ to definable sections of $F$.

The following proposition is straightforward and is left to the reader as an exercise. 
\begin{prop}\label{def vector bundles basics}
Let $M$ be a definable smooth manifold, $E \to M$ a definable vector bundle. Then:
\begin{itemize}
\item[(i)] $E^*, \ \Lambda^kE, \ \operatorname{Sym}^kE$ carry natural definable structures;
\item[(ii)] if $F$ is another definable vector bundle, then $E \o F, \ Hom(E, F), \ E \oplus F$ carry natural definable structures;
\item[(iii)] the trace morphism $\operatorname{tr} \colon E \o E^* \to \R_M$ is definable;
\item[(iv)] a morphism $f \colon F \to E$ is definable if and only if it is definable as a section of $Hom(E, F)$;
\item[(v)]a morphism $f \colon E \to F$ is definable if and only if it determines a definable map $\operatorname{Tot}(E) \to \operatorname{Tot}(F)$. If $\ker f$ (resp. $\coker f$) is a vector bundle, it carries a natural definable structure;
\item[(vi)] if $N$ is a definable manifold and $f \colon N \to M$ is a definable smooth map, then $f^*E$ carries a natural definable structure;
\end{itemize}
\end{prop}

\begin{prop}\label{definable tangent}
Let $M$ be a definable manifold. Let $(U_{\alpha}, \phi_{\alpha})$ be a definable atlas on $M$. Then $$d\phi_{\alpha} \colon TM|_{U_{\alpha}} \to T\phi_{\alpha}(U_{\alpha})=\phi_{\alpha}(U_{\alpha}) \times \R^n$$ defines a definable structure on $TX$, called \emph{the canonical definable structure};
\end{prop}
\begin{proof}
Without loss of generality, we may assume that $\phi_{\alpha}$ are all open embeddings. Then $d\phi_{\alpha}^{-1}(\di/\di x_i)$ defines  trivialisations of $TM|_{U_{\alpha}}$. The rest of the Proposition follows from the fact that the derivatives of definable smooth functions are definable and computations in the local coordinates.
\end{proof}

Using Proposition \ref{def vector bundles basics} one can canonically extend the definable structure on $TM$ to any tensor bundle $(TM)^{\o p} \o (T^*M)^{\o q}$. Without specifying we say that a tensor on a definable manifold is definable if it is definable in the canonical definable structure. In other words, if $E:=(TM)^{\o p} \o (T^*M)^{\o q}$, a tensor $\tau \in \mathcal{C}^{\infty}(M, E)$ is definable if and only if there exists a definable finite open covering $\{U_{\alpha}\}$ on $M$ such that for every $\alpha$ the components of $\tau|_{U_{\alpha}}$ with respect to a definable trivialisation of $E|_{U_{\alpha}}$ can be written as definable functions $\tau^{i_1, \ldots, i_p}_{j_1, \ldots, j_q} \colon U_{\alpha} \to \R$.

All the notions of differential geometry easily transfer to the definable context. For example, one says that a Riemannian metric $g$ on a definable manifold $M$ is definable if $g$ is definable as a section of $\operatorname{Sym}^{2}T^{*}M$. 

We say that a connection $\nabla$ on $E$ is definable if, on a finite covering by definable open charts $U_{\alpha}$ in which both $E$ and $TM$  can be simultaneously definably trivialised, $\nabla$ can be written as $\nabla = d+P_{\alpha}$ for definable local sections $P_{\alpha}$ of $End(E) \o T^*M$. 

\begin{prop}\label{def connection}
Let $M$ be a smooth definable manifold and $E \to M$ a definable smooth vector bundle.
\begin{itemize}
\item[(i)] if $\nabla$ is a definable connection on $E$, then for every definable section $s \in \mathcal{C}^{\infty}(M, E)$ the section $\nabla s \in \mathcal{C}^{\infty}(M , E \o T^{*}M)$ is definable;
\item[(ii)] if $g$ is a definable metric on $M$, its Levi-Civita connection is definable. 
\end{itemize}
\end{prop}
\begin{proof}
\textit{(i)}. Let $s_{\alpha}=s|_{U_{\alpha}}$. Then $\nabla(s_{\alpha})=ds_{\alpha}+P_{\alpha}s_{\alpha}$. Both summands are definable: the second one because of the definability of $P_{\alpha}$, and the first one because the derivative of a definable function is definable (Theorem \ref{o-min nice}, \textit{(iv)}) and $ds_{\alpha}$ can be expressed via partial derivatives of the components of $s_{\alpha}$ with respect to local definable trivialisations.  
\\

\textit{(ii)}. Choose a definable chart $U$ with definable coordinates $(x^1, \ldots, x^n)$. On this chart, $\nabla$ can be written as $d+\Gamma$, where the elements of the tensor $\Gamma$ are Christoffel symbols $\Gamma_{i,j}^k$. Those are expressed in terms of the partial derivatives of $g$ along the basis vector fields: 
\[
\Gamma_{i,j}^k = \frac{1}{2}g^{kl} \left( \frac{\di g_{li}}{\di x^j} + \frac{\di g_{lj}}{\di x^i} - \frac{\di g_{ij}}{\di x^l} \right ).
\]
The functions $g^{a,b}$ are definable with respect to $x^c$ and a derivation of a definable function is definable. Hence $\nabla$ is definable. 

\end{proof}

\begin{df}\label{Weyl inv}
Let $(M, g)$ be a Riemannian manifold. We define the class of \emph{Weyl invariants} among the tensors on $M$ as the following:
\begin{itemize}
\item $g$ is a Weyl invariant;
\item the curvature tensor $R_g$ is a Weyl invariant;
\item if $\tau$ is a Weyl invariant and $\nabla$ is the Levi-Civita connection, $\nabla \tau$ is a Weyl invariant;
\item a linear combination of Weyl invariants is a Weyl invariant;
\item if $\tau \in (TM)^{\o p} \o (T^*M)^{\o q}$ is a Weyl invariant, then its image under any tensorial convolution $$(TM)^{\o p} \o (T^*M)^{\o q} \to (TM)^{\o (p-k)} \o (T^*M)^{\o (q-k)}$$  is a Weyl invariant.
\end{itemize}
\end{df}

Typical examples of Weyl invariants are the Ricci curvature, the Weyl tensor,  the scalar curvature, iterated covariant derivatives of the Riemannian curvature $\nabla^{\circ s}R_g$ et caetera. Combining  Propositions \ref{def vector bundles basics} and \ref{def connection} we get the following result:

\begin{cor}\label{weyl def}
Let $M$ be a definable smooth manifold and $g$ a definable Riemannian metric on it. Then every Weyl invariant on $(M, g)$ is definable.
\end{cor}

\subsection{Additive Weyl invariants and cohomogeneity}\label{Singer part}

Recall that a Riemannian manifold $(M, g)$ is called locally homogeneous if for any two points $x, y \in M$ there exist open neighbourhoods $U, \ V$ of $x$ and $y$ respectively such that $(U, g|_U)$ and $(V, g|_V)$ admit a smooth isometry. Equivalently, this means that for every point $x \in M$ there is an open neighbourhood $U$ of $x$ and a collection of Killing vector fields $\xi_1, \ldots, \xi_n \in \mathcal{C}^{\infty}(U, TM|_U)$ that  generate $TM|_U$ as a $\mathcal{C}^{\infty}(U)$-module. A locally homogeneous Riemannian manifold is homogeneous, that is, admits a transitive Lie group action by isometries, if it is complete and simply connected. 

Let $(M, g)$ be a Riemannian manifold. Let $\operatorname{Fr}_g(M)$ be the bundle of orthonormal frames on $M$. This is a principal $\operatorname{O}(TM)$-bundle endowed with a projection $p \colon \operatorname{Fr}_g(M) \to M$ and $p^*TM$ canonically trivialises. In particular, every Weyl invariant $\tau \in (TM)^{\o p} \o (T^*M)^{\o q}$ defines a vector-valued function  $p^*\tau \colon \operatorname{Fr}_g \to \R^{n^2pq}$, where $n=\dim M$. In 1960 Singer proved (\cite{Sing}) that there exists a constant $k$ depending only on the dimension of $M$, such that $(M,g)$ is locally homogeneous if and only if the maps $p^*(\nabla^sR_g)$ are constant for every $s, 0 \le s \le k$. Later, Pr\"ufer, Tricerri and Vanhecke reduced Singer's criterion to the following elegant form (\cite{PTV}, Theorem 2.3):

\begin{thrm}[Pr\"ufer-Tricerri-Vanhecke]\label{PTV thrm}
Let $(M,g)$ be a Riemannian manifold. There exists a constant $k$ depending only on $\dim M$ and a collection of smooth functions $$w_i \colon M \to \R, \ i = 1, \ldots, k,$$ such that $(M,g)$ is locally homogeneous if and only if each $w_i$ is constant. Moreover, every $w_i$ is a Weyl invariant in the sense of definition \ref{Weyl inv}.
\end{thrm}

Local homogeneity of $(M, g)$  in Theorem \ref{PTV thrm} can be replaced by global, provided that $(M, g)$ is complete and simply connected.

The authors in \cite{PTV} call the functions $w_i$ \emph{additive Weyl invariants}. In fact they all are of the form $\tr(\nabla^{\circ s_1}R_g \o \ldots \nabla^{\circ s_p}R_g)$. If $\dim M = n$, one can take $k=\frac{n(n-1)}{2}$ in both Singer's and Pr\"ufer-Tricerri-Vanhecke theorems. It is natural to gather all additive Weyl invariants into a single map $W_M=(w_1, \ldots, w_k) \colon M \to \R^k$. The meaning of this map was explained by Console and Olmos in \cite{CO}.

\begin{thrm}[Console-Olmos, \cite{CO}]\label{co theorem}
Let $(M,g)$ be a Riemannian manifold and let $$W_M=(w_1, \ldots, w_k) \colon M \to \R^k$$ be as above. Then the following holds:
\begin{itemize}
\item[(i)] every Killing vector field on $(M, g)$ is tangent to the level sets of $W_M$;
\item[(ii)] if $N \subseteq M$ is a non-critical level set of $W_M$, then $(N, g|_N)$ is a locally  homogeneous Riemann manifold;
\item[(iii)] for every $W_M$-non-critical point $x \in M$ there exists a neighbourhood $U$ of $x$ and a collection of Killing vector fields $\xi_1, \ldots, \xi_r \in \mathcal{C}^{\infty}(U, TM|_U)$  that generate $\operatorname{Ker} dW_M$ as an $\mathcal{C}^{\infty}(U)$-module.
\end{itemize}
\end{thrm}

In what follows we refer to the map $W_M$ as \emph{Console-Olmos map}. We point out, that if $(M, g)$ is a definable Riemannian manifold, the map $W_M$ is automatically definable, as follows from Corollary \ref{weyl def}.

\begin{rmk}\label{co singular}
For technical reasons, we will need to apply Console-Olmos theorem not only to smooth manifolds but also to manifolds of $\mathcal{C}^m$-class with $m$ arbitrary large. This difficulty arises only if we do not assume our o-minimal structure $\Sigma$ to be of QA type (see Remark \ref{smooth cell}). This causes no harm since Theorem \ref{co theorem} is still valid for $\mathcal{C}^m$-manifolds with $\mathcal{C}^m$-Riemannian metric if $m \gg k(n)=\frac{n(n-1)}{2}$, where $n=\dim M$ (it is sufficient that $p$-th partial derivatives of the coefficients of the metric in local coordinates exist and are continuous for $p<k(n)$).
\end{rmk}

\section{Riemannian manifolds with definable universal covers}\label{main part}

\subsection{Strongly bi-definable sets and bi-definable functions}\label{bi-def part}

\begin{df}\label{bi-def}
Assume that $\Sigma_1$ and $\Sigma_2$ are o-minimal structures and $\Sigma_2$ is an expansion of $\Sigma_1$. Let $M_1$ (respectively $M_2$) be a $\Sigma_1$-definable (respectively $\Sigma_2$-definable) topological space. Let $\phi \colon M_2 \to M_1$ be a continuous map. We say that a subset $S\subseteq M_1$ is \emph{strongly bi-definable} if  $S$ is definable in $M_1$ and $\phi^{-1}(S)$ is definable in $M_2$. We say that a map $f \colon M_1 \to T$ to a $\Sigma_1$-definable space $T$ is \emph{bi-definable} if $f$ is $\Sigma_1$-definable and $\phi^*f \colon M_2 \to T$ is $\Sigma_2$-definable.
\end{df}

Our definition of strongly bi-definable subsets is stronger than the notion of \emph{bi-definable subsets} introduced by Klingler, Ullmo and Yafaev in \cite{KUY}. Roughly speaking, the authors in \cite{KUY} only require a single connected component of $\phi^{-1}(S)$ to be definable.

From now we will always assume that $M_1=M$ is a compact smooth manifold and $M_2=\tM$ is its universal cover. The map $\phi=\pi \colon \tM \to M$ is the covering map. We will not mention explicitly the structures $\Sigma_1$ and $\Sigma_2$. In order to make the discussion more intrinsic, one can assume that $\Sigma_1=\R_{\an}$, since every compact smooth manifold is naturally endowed with a structure of a $\R_{\an}$-definable manifold. Therefore, our only assumption is that its universal cover $\tM$ is endowed with a structure of a $\Sigma$-definable smooth manifold, where $\Sigma=\Sigma_2$ is an o-minimal expansion of $\R_{\an}$.

It is easy to check that finite unions, finite intersections and complements of strongly bi-definable sets are strongly bi-definable.

The following proposition shows that strongly be-definable subsets share very non-trivial topological properties.

\begin{prop}\label{bi-def top}
Let $i \colon N \hookrightarrow M$ be a  strongly bi-definable subset. Let $\widetilde{N}:=\pi^{-1}(N)$. The following holds:
\begin{itemize}
\item[(1)] for any ring R the cohomology algebra $H^{\bullet}(\widetilde{N}, R)$ is finitely generated;
\item[(2)] The image of $i_*\colon \pi_1(N) \to \pi_1(M)$ is of finite index in $\pi_1(M)$.
\end{itemize}
\end{prop}
\begin{proof}
The item \textit{(i)} follows from the finite cell decomposition for $\widetilde{N}$ (Theorem \ref{o-min nice}, \textit{(iii)}). To see \textit{(ii)} note that \textit{(i)} implies that $\widetilde{N}$ has finitely many connected components and the connected components of $\widetilde{N}$ are in bijection with $\pi_1(M)/i_*(\pi_1(N))$.
\end{proof}

\subsection{Proof of Theorem A}\label{proof a part}

\begin{thrm}\label{thmA}
Let $M$ be a compact real analytic manifold and $g$ a real analytic Riemann metric on $M$. Let $\pi \colon \tM \to M$ be the universal cover and $\widetilde{g}:=\pi^*g$.  Assume that $\tM$ is a definable manifold and that $\pi^*g$ is a definable Riemannian metric on it. Then there exists a  compact simply-connected definable space $\widetilde{K}$ and a smooth definable map $\Phi \colon \tM \to \widetilde{K}$ such that every non-critical level set of $\Phi$ is a locally homogeneous Riemannian manifold with respect to the metric $\widetilde{g}$. Moreover, if $x \in \tM$ is a non-critical point, the level sets of $\Phi$ in a neighbourhood $U$ of $x$ are given by orbits of a connected Lie group locally acting by isometries on $U$. 
\end{thrm}

\begin{rmk}
One should clarify what is meant by non-critical points and non-critical level sets of $\Phi$, since a priori $\Phi$ is just a map to a topological space. In fact, the target $\widetilde{K}$ of $\Phi$ is definable and by Remark \ref{smooth cell} there exists a dense open definable subset $U \subset \widetilde{K}$ which is a definable $\mathcal{C}^2$ manifold. Therefore the differential $d\Phi$ is well-defined on $\Phi^{-1}(U)$. We say that $x \in M$ is critical (respectively  that a level set $N = \Phi^{-1}(y), \ y \in \widetilde{K},$ is critical), if $\Phi(x) \in U$ and $d\Phi|_{x}$ is not surjective (respectively if $y \in U$ and $d\Phi|_{x}$ is not surjective for some $x \in N$).
\end{rmk}

\begin{proof}[Proof of Theorem \ref{thmA}]
Let $W_{\tM}$ to be the Console-Olmos map of $(\tM, \widetilde{g})$. Let $K \subseteq \R^k$ be the image of $W_{\tM}$. By Corollary \ref{weyl def} the map $W_{\tM}$ is definable. In particular, $K$ is definable in $\R^k$ and the fibres of $W_{\tM}$ are definable subspaces of $\tM$. 

Notice that $W_{\tM}$ coincides with the pullback $\pi^*W_{M}$ of the Console-Olmos map of $(M,g)$. Indeed, $\pi \colon (\tM, \widetilde{g}) \to (M, g)$ is a local isometry, therefore each point $x \in \tM$ has a neighbourhood $U$, such that $W_{\tM}|_U = (\pi^*W_M)|_U$. Hence $K= W_M(M)$  is compact since $M$ is compact.

Let $p \colon \widetilde{K} \to K$ be the universal cover of $K$ and $\Phi \colon \tM \to \widetilde{K}$ be the canonical lifting of $W_{\tM}$. We get a diagram:
\[
\xymatrix{
\widetilde{M} \ar[rr]^{\Phi}  \ar[rrd]^{W_{\tM}} \ar[d]_{\pi} && \widetilde{K} \ar[d]^{p} \\
M \ar[rr]_{W_M} && K
}
\]
Since $W_{\tM}$ is definable, its fibres have finitely many connected components. At the same time, $W_{\tM}=p \circ \Phi$, so the fibres of $p$ have finitely many connected components as well. This means that $\pi_1(K)$ is finite and by Lemma \ref{o-min cover} the topological space $\widetilde{K}$ can be endowed with a definable structure in such a way that $\Phi \colon \tM \to \widetilde{K}$ is definable. Since  $\widetilde{K}$ is a finite cover of $K$ and $K$ is compact, $\widetilde{K}$ is compact as well. The rest of the claim follows immediately from Console-Olmos theorem (Theorem \ref{co theorem}) 
\end{proof}

The conclusion of Theorem \ref{thmA} holds only for non-critical points and this limitation cannot be ruled out. Below we present two examples when the map $\Phi$ admits critical fibres.

\begin{ex}\label{revolution}

 Let $g$ be  a Riemannian metric on the two-dimensional sphere $S^2$ such that $$\operatorname{Isom}(S^2, g) \simeq \operatorname{U}(1).$$ 

More explicitly, one can take $f \colon [-1;1] \to \R$ to be an analytic function with 
\[
\begin{cases}
f(-1)=f(1)=0\\
f(t)>0, \ -1 <t <1.
\end{cases}
\]
Let $(M, g)$ be the surface of revolution of the graph of $f$ inside $\R^3$ with the induced metric. Then $M$ is diffeomorphic to $S^2$ and if $f$ is generic enough, the only isometries of $(M,g)$ are given by rotations around the $z$-axis.

The Console-Olmos map in this case is the projection onto the $z$-axis $$W_M \colon M \to [-1;1].$$ The non-critical fibres $W_M^{-1}(t), t \neq \pm 1$ are isometric to $S^1$ with the standard homogeneous metric (up to scaling), while the critical fibres $W_M^{-1}(-1)$ and $W_M^{-1}(1)$ are just points. 

Notice that $M$ is simply connected, therefore $(\widetilde{M}, \widetilde{g})=(M, g)$ and $(M, g)$ admits a natural structure of a definable Riemannian manifold  in $\R_{an}$.
\end{ex}

\begin{ex}\label{revolution2}
The previous example formally satisfies the assumptions of Theorem \ref{mainA}, although not very interesting from our perspective, since in this case $\pi_1(M)=\{1\}$. Here is a way to modify this example in order to get examples with interesting fundamental groups. 

Let $S$ be a two-dimensional sphere and $g_S$ be a metric on it as in Example \ref{revolution}. Let $(M_0, g_0)$ be a compact locally homogeneous Riemannian manifold. Take $M:= M_0 \times S$ with the product metric $g=(g_0 \times g_S)$.

The universal cover $\widetilde{M_0}$ of $M_0$ can be identified with a homogeneous space $G/K$. The proof of \cite{BKT}, Lemma 2.1 can be easily adapted to show that $\widetilde{M_0}=G/K$  carries canonical $\R_{\alg}$-definable structure. The metric $g_0$ pullbacks to a semi-algebraic homogeneous metric $\widetilde{g_0}$ on $\widetilde{M_0}$. This implies that  the universal cover of $(M, g)$ is isometric to $(\widetilde{M_0} \times S, \widetilde{g_0}\times g_S)$ and is an $\R_{an}$-definable Riemannian manifold. 

The Console-Olmos map for $(M, g)$ is the composition of the projection on $S$ and the Console-Olmos map for $S$:
\[
\xymatrix{
M=M_0 \times S \ar[r]^{p_S} \ar[rd]_{W_M} & S \ar[d]_{W_S}\\
& [-1;1].
}
\]
In particular the non-critical fibres of $W_M$ are isometric to locally homogeneous manifolds $M_0 \times S^1$, while the critical fibres $W_M^{-1}(\{-1\})$ and $W_M^{-1}(\{1\})$ are isometric to $(M_0, g_0)$. One can further modify the geometry of $M$ such that it does not split as a direct product anymore: instead of taking a product metric $g_0 \times g_S$ one can take a skew-product metric $(p^*_Se^{\phi})g_0 \times g_S$ where $\phi \colon S \to \R$ is a real-analytic function.
\end{ex}

Examples \ref{revolution} and \ref{revolution2} suggest that there is no way to get rid of the singularities of the map $\Phi$. Nevertheless, one can find a stratification of $M$ by strongly bi-definable locally closed subsets each of which admits a smooth fibration by locally homogeneous manifolds. This is the content of the following Proposition.

\begin{prop}\label{strata}
Let $\Sigma$ be an o-minimal structure and $(M, g)$ a $\Sigma$-definable Riemannian manifold of class $\mathcal{C}^m$ for $m\gg \frac{n(n-1)}{2}$, where $n=\dim M$. Let $\pi \colon \widetilde{M} \to M$ be a cover of $M$ and $\widetilde{g}:=\pi^*g$. Assume that $(\widetilde{M}, \widetilde{g})$ is a definable Riemannian manifold. Then for every $k, \ m> k >0$ there exists a stratification of $M$ by locally closed strongly bi-definable subsets 
\[
\varnothing = S_0 \subset \ldots \subset S_r=M,
\]
such that each $V_i:= S_i \setminus S_{i-1}$ is a $\mathcal{C}^k$-submanifold of $M$ that admits definable locally trivial fibrations $f_i \colon V_i \to K_i$ with locally homogeneous fibres onto a definable topological space $K_i$. If, moreover, $\Sigma$ is of QA type, the same holds for $k=\infty$.
\end{prop}

Notice that we do not require $(M,g)$ to be compact anymore. We also do not ask that $\widetilde{M} \to M$ be universal.

\begin{proof}[Proof of Proposition \ref{strata}]
The proof is done by induction on $\dim M$. If $\dim M=0$, the statement is clear. Otherwise, let $W_M \colon M \to \R^n$ be the Console-Olmos map of $(M,g)$ and $K=W_M(M)$. This map exists (see Remark \ref{co singular}) and its regular fibres are locally homogeneous. It is moreover bi-definable with respect to the map $\widetilde{M} \to M$ since the same argument as in the proof of Theorem \ref{thmA} applies.

Let  $K^{\circ} \subset K$ be the set of non-critical values of $W_M$. Clearly, $K^{\circ}$ is an open dense definable subset and $M^{\circ}:=W_M^{-1}(K^{\circ})$ is an open dense definable submanifold  of $M$. In fact, $M^{\circ}$ is strongly bi-definable since $\pi^{-1}(M^{\circ})=W_{\widetilde{M}}^{-1}(K^{\circ})$ is definable in $\widetilde{M}$. 

Let $N:= M \setminus M^{\circ}$. This is a closed strongly bi-definable subset of $M$. Notice that $\dim N < \dim M$. In fact, $N=S_{r-1}$ will be the first step of our stratification.

By Lemma \ref{smooth strata} we can find a stratification of $N$ by definable locally closed subsets $N_i, \ \varnothing=N_0 \subset \ldots \subset N_s =N$.  

We claim that in fact, we can choose all $N_i$ to be strongly bi-definable. Indeed, one can choose $N_{s-1}$ to be the singular locus of $N$. Clearly $\pi^{-1}(N_{s-1})$ is the singular locus of the definable subset $\pi^{-1}(N) \subset \widetilde{M}$, hence $N_{s-1}$ is strongly bi-definable  and we proceed by induction on $\dim N$.

Now for every $i$ the set $U_i:=N_i\setminus N_{i-1}$ is a definable Riemannian $\mathcal{C}^k$-manifold with definable cover $\pi_i \colon \widetilde{U_i} = \pi^{-1}(U_i) \to U_i$ and $\pi_i^*g=\widetilde{g}|_{\widetilde{U_i}}$ is definable. Finally, notice that $\dim U_i \le \dim N < \dim M$. Therefore, we may apply the induction assumption to $U_i$ and for each $U_i$ get stratifications
\[
\varnothing = S^i_0 \subset \ldots \subset S^i_{r_i}
\]
with the required properties. It is left to gather all nested stratifications $S^i_j$ into a single stratification $S_p$ of $M$. For this one can put $S_p:= \bigcup_{\dim S^i_j =n_p} S^i_j$, where $\{n_p\}$ is the set of natural numbers appearing as dimensions of the strata of $S^i_j$.

For each $(V_i:= S_i \setminus S_{i-1}, g|_{V_i})$ the Console-Olmos map $f_i:=W_{V_i} \colon V_i \to K_i$ has no critical fibres, therefore defines a locally trivial fibration with locally homogeneous fibres.
\end{proof}

\begin{cor}\label{union}
In the assumptions of Theorem \ref{thmA} the manifold $M$ can be written as a disjoint union of strongly bi-definable locally homogeneous manifolds.
\end{cor}
\begin{proof}
Take the union of all the fibres of the maps $f_i$, where $f_i$ are as in Proposition \ref{strata}.
\end{proof}

Another Corollary of Theorem \ref{thmA} is that the definability of the Riemannian universal cover of a Riemannian manifold $(M,g)$ puts certain topological restrictions on $M$, namely on its fundamental group. 

\begin{cor}\label{quasi-isometric}
Let $\Sigma$ be an o-minimal structure and $(U, g)$  a $\Sigma$-definable Riemannian manifold. Let $\Gamma$ be a finitely generated group. Assume that $\Gamma$ acts on $(U, g)$ by isometries and the action is free, cocompact and properly discontinuous. Then $\Gamma$ is quasi-isometric to a locally homogeneous space.
\end{cor}
\begin{proof}
Let $M:=U/\Gamma$ and $p \colon U \to M$ the projection. The metric $g$ descends to a Riemannian metric $\overline{g}$ on $M$. Since $U \to M$ is a $\Gamma$-Galois cover, we have a surjective homomorphism $\pi_1(M) \to \Gamma$.

Let $w_M \colon M \to K$ be the Console-Olmos map on $(M, \overline{g})$ and $F \subset M$ a connected component of a non-critical fibre of $w_M$. Denote by $\Gamma_F$ the image of the composition homomorphism $\pi_1(F) \to \pi_1(M) \to \Gamma$. The connected components of $p^{-1}(F)$ are in bijection with $\Gamma/\Gamma_F$. On the other hand, $p^{-1}(F)$ consists of several connected components of a fibre of the Console-Olmos map $w_U \colon (U, g) \to K'$, hence definable. It follows that $\pi_0(p^{-1}(F)) = \Gamma/\Gamma_F$ is finite. Therefore, $\Gamma$ is quasi-isometric to $\Gamma_F$.

Let $\widetilde{F}$ be a connected component of $p^{-1}(F)$. The group $\Gamma_F$ acts on $(\widetilde{F}, g|_{\widetilde{F}})$ by isometries and the action is free and properly discontinuous with a compact quotient $F=\widetilde{F}/\Gamma_F$. By Milnor-\v{S}varc Lemma (\cite{BH}, Proposition 8.19) this implies $\Gamma_F$ is quasi-isometric to $(\widetilde{F}, g|_{\widetilde{F}})$. By the Console-Olmos Theorem (Theorem \ref{co theorem}) $(\widetilde{F}, g|_{\widetilde{F}})$ is locally homogeneous.
\end{proof}

Notice that being quasi-isometric to a locally homogeneous space seems to be a subtle property of a finitely generated group. For instance a group $\Gamma$ as in Corollary \ref{quasi-isometric} can not be of intermediate growth. 

\subsection{The case of aspherical $M$}\label{proof b part} Now we want to prove Theorem \ref{mainB}, which strengthens Theorem \ref{mainA} in the case of aspherical $M$. The main idea behind the proof is that if $M$ is aspherical, it can have no strict closed submanifolds which are strongly bi-definable. On the other hand, the smooth fibres of the Console-Olmos map should provide such submanifolds, which forces the Console-Olmos map to be constant.

\begin{prop}\label{kill the action}
Let $M$ be a finite-dimensional compact connected CW-complex and $N \subseteq M$ a closed connected subcomplex. Let $\pi \colon \tM \to M$ be the universal cover and $\widetilde{N}:=\pi^{-1}(N)$ the preimage of $N$ in it. Assume that the cohomology algebra $H^{\bullet}(\widetilde{N}, \Z/2\Z)$ is finitely generated. Then there exists a finite connected cover $q \colon M' \to M$ and a connected subcomplex $N' \subseteq M', \ q(N') = N$, such that the action of $\pi_1(N')$ on $H^{\bullet}(\widetilde{N}, \Z/2\Z)$ is trivial.
\end{prop}
\begin{proof}
Note that we do not assume $\widetilde{N}$ to be connected. Nevertheless, $\widetilde{N}$ has only finitely many connected components since its cohomology algebra is finitely generated, and $$\pi_0(\widetilde{N})=\pi_1(M)/i_*\pi_1(N)$$ is finite (here $i\colon N \hookrightarrow M$ is the tautological embedding).

Let $\Gamma:=\pi_1(M)$. Since $\widetilde{N}$ is a $\Gamma$-Galois cover of $N$, the action of $\pi_1(N)$ on $H^{\bullet}(\widetilde{N}, \Z/2\Z)$ factors through the natural homomorphism $\pi_1(N) \to \pi_1(M)=\Gamma$. At the same time, if $H^{\bullet}(\widetilde{N}, \Z/2\Z)$ is finitely generated, it is a finite-dimensional $\Z/2\Z$-algebra and the group $\operatorname{Aut}(H^{\bullet}(\widetilde{N}, \Z/2\Z)$ is finite. Let $G$ be the image of $\pi_1(N)$ in $\operatorname{Aut}(H^{\bullet}(\widetilde{N}, \Z/2\Z))$. Consider the composition of homomorphisms of groups 
$$
\pi_1(N) \to \Gamma \to G \hookrightarrow \operatorname{Aut}(H^{\bullet}(\widetilde{N}, \Z/2\Z)).
$$ Let $q \colon M' \to M$ be the finite $G$-Galois cover of $M$ associated to the group homomorphism $\pi_1(M)=\Gamma \to G \subseteq \operatorname{Aut}(H^{\bullet}(\widetilde{N}, \Z/2\Z))$. Let $N':=q^{-1}(N)$. Notice that $N'$ is connected since both $\pi_1(N)$ and $\pi_1(M)$ map to $G$ surjectively. Of course the universal cover $\tM \xrightarrow{\pi} M$ factors through $M' \xrightarrow{q} M$ and $\widetilde{N} \to N$ factors through $N' \to N$. At the same time, the action of $\pi_1(N')$ on $H^{\bullet}(\widetilde{N}, \Z/2\Z)$ is trivial.
\end{proof}

\begin{lemma}\label{podkorytov}
Let $\Gamma$ be a finitely generated group and $M$ a compact manifold which is a $K(\Gamma, 1)$. Let $N \subseteq M$ be a closed submanifold. Let $\pi \colon \tM \to M$ be the universal cover and $\widetilde{N}:=\pi^{-1}(N)$. Assume that  $H^{\bullet}(\widetilde{N}, \Z/2\Z)$ is finitely generated. Then $N=M$.
\end{lemma}
\begin{proof}
First of all, notice that we can replace $M$ and $N$ with $M'$ and $N'$ as in Proposition \ref{kill the action} so that the action of $\pi_1(N)$ on $H^{\bullet}(\widetilde{N}, \Z/2\Z)$ is trivial. Indeed, $M'$ will be still aspherical, will have the same universal cover, and $N'=M'$ if and only if $N=M$. Therefore, starting from now we assume that $\pi_1(N)$ acts trivially on the cohomology algebra of $\widetilde{N}$.

Since $\widetilde{N}$ is a $\Gamma$-Galois cover of $N$, the $\Z/2\Z$-cohomology groups of $N$ can be computed using the Cartan-Leray  spectral sequence (\cite{McCl}, Theorem $8^{bis}.9$):
\[
E_2^{i,j}=H^i(\Gamma, H^j(\widetilde{N}, \Z/2\Z)) \implies H^{i+j}(N, \Z/2\Z).
\]
Since we assume that the $H^{j}(\widetilde{N}, \Z/2\Z)$'s are trivial $\pi_1(N)$-modules and the action of $\pi_1(N)$ on $\widetilde{N}$ factors through $\Gamma$, each $H^j(\widetilde{N}, \Z/2\Z)$ is  a finite-dimensional $\Z/2\Z$-vector space with trivial $\Gamma$-action and $H^i(\Gamma, H^j(\widetilde{N}, \Z/2\Z))=H^i(\Gamma, \Z/2\Z) \o_{\Z/2\Z}H^j(\widetilde{N}, \Z/2\Z)$. Let $n:=\dim M$. The compact manifold $M$ is a $K(\Gamma, 1)$ space, thus $H^n(\Gamma, \Z/2\Z)=\Z/2\Z$ and $H^{n+i}(\Gamma, \Z/2\Z)=0$ for $i>0$. 

Let $r$ be maximal integer such that $H^r(\widetilde{N}, \Z/2\Z)$ is non-zero. Consider $$E_2^{n,r}=H^n(\Gamma, H^r(\widetilde{N}, \Z/2\Z).$$ It is non-zero, since it equals
\[
H^n(\Gamma, \Z/2\Z) \o H^r(\widetilde{N}, \Z/2\Z)=H^r(\widetilde{N}, \Z/2\Z).
\]
At the same time  $$E_2^{n-2,r+1}=H^{n-2}(\Gamma, H^{r+1}(\widetilde{N}, \Z/2\Z))=0$$ and $$E_2^{n+2, r-1}=H^{n+2}(\Gamma, H^{r-1}(\widetilde{N}, \Z/2\Z)=0.$$ Therefore $E_2^{n,r}$ survives in the spectral sequence and $H^{n+r}(N, \Z/2\Z)$ is non-zero. We assumed that $N$ is a closed manifold, so $H^p(N, \Z/2\Z)$ vanishes for $p>\dim N$. It follows that $n+r \le \dim N$. Of course $\dim N \le \dim M=n$ and it follows that $n+r=\dim N$, that is, $\dim N=\dim M$ and $r=0$. Since $M$ was assumed to be connected and $N$ is a closed submanifold, we deduce that $N=M$.
\end{proof}

\begin{cor}\label{no bi-def}
Let $M$ be a compact aspherical real analytic manifold and $\pi \colon \widetilde{M} \to M$ its universal cover. Assume that $\widetilde{M}$ admits a structure of a manifold definable in an o-minimal structure. Then $M$ has no strongly bi-definable strict submanifolds.
\end{cor}
\begin{proof}
If $N \subseteq M$ is strongly be-definable, $\pi^{-1}(N)=:\widetilde{N} \subseteq \widetilde{M}$ is definable. In particular, $H^{\bullet}(\widetilde{N}, \Z/2\Z)$ is finitely generated by Theorem \ref{o-min nice}, \textit{(ii)}. The Corollary then follows from Lemma \ref{podkorytov}.
\end{proof}

\begin{thrm}\label{homogene}
Let $(M,g)$ be a compact analytic Riemannian manifold and $\pi \colon \widetilde{M} \to M$ its universal cover. Assume that $\widetilde{M}$ admits a structure of a manifold definable in an o-minimal structure and the metric $\widetilde{g}:= \pi^*g$ is definable. Assume further that $\widetilde{M}$ is contractible, that is, $M$ is aspherical. Then $(\widetilde{M}, \widetilde{g})$ is a Riemannian homogenous space.
\end{thrm}
\begin{proof}
Let $W_{M} \colon M \to \R^k$ be the Console-Olmos map. Let $N \subseteq M$ be a connected component of a non-critical level set of $W_M$. Then $N$ is a smooth submanifold of $M$ and in the same time, it is strongly bi-definable: $\pi^{-1}(N)$ is a union of several connected components of a level set of the Console-Olmos map $W_{\widetilde{M}}$ of $(\widetilde{M}, \widetilde{g})$. However, by Corollary \ref{no bi-def} this implies $N=M$ and the map $W_M$ is constant. Since $(\widetilde{M}, \widetilde{g})$ is complete and simply connected, it is homogeneous by Pr\"ufer-Tricerri-Vanhecke theorem (Theorem \ref{PTV thrm}).
\end{proof}

\begin{cor}
Let $(M, g)$ be a compact Riemannian manifold with a non-positive sectional curvature. Let $\pi \colon \widetilde{M} \to M$ be the universal cover and $\widetilde{g}:=\pi^*g$. Assume that $\widetilde{M}$ admits a structure of a $\Sigma$-definable manifold for some o-minimal structure $\Sigma$ such that $\widetilde{g}$ is $\Sigma$-definable. Then  $(M, g)$ is locally homogeneous. 
\end{cor}
\begin{proof}
The universal cover $\widetilde{M}$ is contractible by Cartan-Hadamard Theorem (\cite{BH}, Ch. 4), so the statement follows from Theorem \ref{homogene}.
\end{proof}

\subsection{Applications to the classical Koll\'ar-Pardon problem}\label{kp part}

Using Theorem \ref{homogene} we prove the original Koll\'ar-Pardon Conjecture in the special case where $X$ is aspherical and possesses a bi-definable K\"ahler metric.

\begin{cor}\label{KP special}
Let $X$ be a smooth algebraic variety and $\pi \colon \widetilde{X} \to X$ be its universal cover. Assume that $\widetilde{X}$ admits a structure of a definable complex manifold in some o-minimal structure. Assume moreover, that  $X$ is aspherical and admits a \emph{bi-definable K\"ahler metric}, that is, a K\"ahler Hermitian metric $h=g + \sqrt{-1}\omega$ such that $\pi^*h$ is definable on $\widetilde{X}$. Then $(\widetilde{X}, \pi^*h)$ is a K\"ahler homogeneous manifold. In particular,  Conjecture \ref{KP conj} holds for $X$.
\end{cor}
\begin{proof}
Let $\widetilde{h}=\pi^*h$ and $\widetilde{h}=\widetilde{g}+\sqrt{-1}\widetilde{\omega}$. Here $\widetilde{g}$ is a Riemannian metric on $\widetilde{X}$ which is the real part of $\widetilde{h}$, and $\widetilde{\omega}$ is a symplectic form which is the imaginary part of $\widetilde{h}$. Being the real part of a definable tensor, $\widetilde{g}$ is definable. Theorem \ref{homogene} implies that $(\widetilde{X}, \widetilde{g})$ is homogeneous as a Riemannian manifold. Let $\xi$ be a Killing vector field on $\widetilde{X}$. Since the form $\widetilde{\omega}$ is harmonic with respect to $\widetilde{g}$, its Lie derivative along $\xi$ vanishes. This implies that any one-parameter subgroup of diffeomorphisms of $\widetilde{X}$ that preserves $\widetilde{g}$, preserves also $\widetilde{\omega}$. Hence it preserves the complex structure operator $J_{\widetilde{X}}$ as well. We deduce that if $G$ is a Lie group that acts by isometries and transitively on $\widetilde{X}$, its action preserves the complex structure. Therefore, $\widetilde{X}$ is homogeneous not only as Riemannian but also as a complex (and as a K\"ahler) manifold.
\end{proof}

\subsection{Open questions}\label{questions}

It seems that one can consider a Koll\'ar-Pardon-type Conjecture with regard to any geometric structure. Informally speaking, the corresponding conjecture should sound as follows. Let $M$ be a compact smooth manifold and $\pi \colon \widetilde{M} \to M$ be its universal cover. Assume that $\widetilde{M}$ is definable in some o-minimal structure and $M$ is endowed with a \emph{geometric structure} which gives rise to a definable geometric structure on $\tM$. What can be said about the geometry of $\widetilde{M}$? Under which conditions is it true that there exists a Lie group action on $\tM$ with compact Hausdorff quotient which preserves this geometric structure?

We do not make here the notion of a geometric structure precise, although there are different ways of doing this, see e.g. \cite{KN}. One can ask the following na\"ive version of the same question:

\begin{q} \label{quest1}
Let $M$ be a compact smooth manifold and $\tau$ a tensor on $M$. Let $\pi \colon \tM \to M$ be its universal cover. Assume that $\tM$ admits a definable structure such that $\pi^*\tau$ is definable. Is it true that there exists a definable simply connected space $K$ and a definable smooth map $\Phi \colon \tM \to K$, such that $\operatorname{Ker} d\Phi$ is locally generated by vector fields preserving $\tau$? 
\end{q}

\begin{q}\label{quest2}
In the assumptions of Question \ref{quest1}, is it true that if $\tM$ is contractible, there exists a Lie group $G$ acting transitively on $\tM$ and preserving $\tau$?
\end{q}

We do not expect this to hold for an arbitrary tensor $\tau$, although we are not aware of counterexamples. If $\tau$ is a Riemannian metric, the affirmative answer is given by our Theorems \ref{mainA} and \ref{mainB}.

If $\tau$ is an integrable complex structure, Koll\'ar-Pardon conjecture is a particular case of Question \ref{quest1}, but Question \ref{quest1} is much more general in this case. For example, a Hopf surface $H=(\C^2 \setminus \{0\})/\Z$ is a compact complex manifold which is not algebraic, although its universal cover is definable. Notice that the methods of \cite{KP} and \cite{CHK} are based on deep algebro-geometric and Hodge-theoretic techniques and fail in the non-K\"ahler setting.

It would be interesting to study the same question in the cases when $\tau$ is a symplectic structure or an integrable para-complex structure.

Notice that once the affirmative answer to Question \ref{quest1} is established, Corollary \ref{no bi-def} together with the arguments of the proof of Theorem \ref{homogene} imply the affirmative answer to  Question \ref{quest2}.

Another possible direction of generalisation is to consider a  manifold $M$, which is not necessarily compact, but definable in some o-minimal structure (this motivates the generality that we chose in Definition \ref{bi-def}). The first step in this direction would be to find an analogue of the Claudon-H\"oring-Koll\'ar Conjecture \cite{CHK} and to understand smooth complex quasi-projective varieties with algebraic universal cover. We are going to address this question in our subsequent work \cite{Rog}.

\end{document}